\documentclass{amsart}
\usepackage{amsfonts,amssymb,amscd,amsmath,enumerate,verbatim,calc}
\headsep=1cm \numberwithin{equation}{section}
\newtheorem{thm}{Theorem}[section]
\newtheorem{cor}[thm]{Corollary}
\newtheorem{lem}[thm]{Lemma}
\newtheorem{prop}[thm]{Proposition}
\newtheorem{defn}[thm]{Definition}

\newtheorem{rem}[thm]{Remark}
\newtheorem{ques}[thm]{Question}

\newcommand{\Ann}{\mbox{Ann}\,}

\newcommand{\Ext}{\mbox{Ext}\,}

\newcommand{\Spec}{\mbox{Spec}\,}

\newcommand{\Ass}{\mbox{Ass}\,}
\newcommand{\Assh}{\mbox{Assh}\,}
\newcommand{\Att}{\mbox{Att}\,}
\newcommand{\Supp}{\mbox{Supp}\,}
\newcommand{\hSupp}{\mbox{hSupp}\,}

\renewcommand{\dim}{\mbox{dim}\,}

\newcommand{\cd}{\mbox{cd}\,}

\newcommand{\Min}{\mbox{Min}\,}

\newcommand{\h}{\mbox{ht}\,}
\newcommand{\im}{\mbox{Im}\,}
\newcommand{\E}{\mbox{E}}

\renewcommand{\H}{\mbox{H}}
\newcommand{\V}{\mbox{V}}

\newcommand{\ara}{\mbox{ara}}
\newcommand{\fa}{\mathfrak{a}}
\newcommand{\fb}{\mathfrak{b}}

\newcommand{\fm}{\mathfrak{m}}
\newcommand{\fp}{\mathfrak{p}}
\newcommand{\fq}{\mathfrak{q}}

\newcommand{\fc}{\mathfrak{c}}

\begin{document}

\title[Artinian and non-artinian local cohomology modules]
 {Artinian and non-artinian local cohomology modules}

\bibliographystyle{amsplain}

     \author[Mohammad T. Dibaei]{Mohammad T. Dibaei$^{1}$}
     \author[A. Vahidi]{A. Vahidi$^{2}$}

\address{$^{1, 2}$ Faculty of Mathematical Sciences and Computer,
Tarbiat Moallem University, Tehran, Iran.}

\address{$^1$ School of Mathematics, Institute for Research in Fundamental Sciences (IPM),
P. O. Box: 1935-5746, Tehran, Iran.}
\email{dibaeimt@ipm.ir}

\address{$^2$ Payame Noor University (PNU), Iran.}
\email{vahidi.ar@gmail.com}

\thanks{M. T. Dibaei was supported by a grant from IPM No. ....}




\begin{abstract}
Let $M$ be a finite module over a commutative noetherian ring $R$.
For ideals $\fa$ and $\fb$ of $R$, the relations between
cohomological dimensions of $M$ with respect to $\fa, \fb$,
$\fa\cap\fb$ and $\fa+ \fb$ are studied. When $R$ is local, it is
shown that $M$ is generalized Cohen-Macaulay if there exists an
ideal $\fa$ such that all local cohomology modules of $M$ with
respect to $\fa$ have finite lengths. Also, when $r$ is an integer
such that $0\leq r< \dim_R(M)$, any maximal element $\fq$ of the
non-empty set of ideals $\{\fa$ : $\H_\fa^i(M)$ is not artinian for
some $i$, $i\geq r$$\}$ is a prime ideal and that all Bass numbers
of $\H_\fq^i(M)$ are finite for all $i\geq r$.
\end{abstract}

\maketitle

\section{Introduction}

Throughout $R$ is a commutative noetherian ring, $\fa$ is a proper
ideal of $R$, $X$ and $M$ are non--zero $R$--modules and $M$ is a
finite (i.e. finitely generated). Recall that the $i$th local
cohomology functor $\H_\fa^i$ is the $i$th right derived functor of
the $\fa$--torsion functor $\Gamma_\fa$. Also, the cohomological
dimension of $X$ with respect to $\fa$, denoted by $\cd(\fa, X)$, is
defined as $$\cd(\fa, X):= \sup \{i: \H^{i}_{\fa}(X)\neq 0\}.$$

In section 2, we discuss the arithmetic of cohomological dimensions.
We show that the inequalities  $\cd(\fa+ \fb,M)\leq \cd(\fa, M)+
\cd(\fb, M)$ and $\cd(\fa+ \fb, X)\leq \ara(\fa)+ \cd(\fb, X)$ hold
true and  we find some equivalent conditions for which each
inequality becomes equality.

In section 3, we study artinian local cohomology modules. We first
observe that over a local ring $(R, \fm)$ if there is an integer $n$
such that $\dim_{R}(\H_\fa^i(X))\leq 0$ for all $i\leq n$
(respectively, for all $i\geq n$), then $\H^{i}_{\fa}(X)\cong
\H^{i}_{\fm}(X)$ for all $i\leq n$ (respectively, for all $i\geq n+
\ara(\fm/\fa$)) (Theorem \ref {3--2}). In this situation, if $X$ is
finite then $\H_\fa^i(X)$ is artinian for all $i\leq n$
(respectively, for all $i\geq n+ \cd(\fm/\fa, X)$), which is related
to the third of Huneke's four problem in local cohomology \cite{Hu}.
Here, for ideals $\fa\subseteq\fb$, $\cd(\fb/\fa, X)$ is introduced
to be the infimum of the set $\{\cd(\fc, X): \fc $ is an ideal of
$R$ and $\sqrt{\fb}= \sqrt{\fc+ \fa}\}.$ It is deduced that $M$ is
generalized Cohen-Macaulay if there exists an ideal $\fa$ such that
all local cohomology modules of $M$ with respect to $\fa$ have
finite lengths (Corollary \ref{3--4}).

Section 4 is devoted to study non-artinian-ness of local cohomology
modules. Note that $\cd(\fa+ Rx, X)\leq \cd(\fa, X)+ 1$ for all
$x\in R$ \cite[Lemma 2.5]{DNT}, we show that if there exist
$x_{1},...,x_{n}\in R$ such that $\cd(\fa+ (x_{1},...,x_{n}), X)=
\cd(\fa, X)+ n$, then $\dim_R(\H_\fa^{\tiny{\cd(\fa, X)}}(X))\geq n$
and so $\H_\fa^{\tiny{\cd(\fa, X)}}(X)$ is not artinian (Corollary
\ref {4--1}). For each integer $r$, $0\leq r< d$ ($ d:= \dim_R(M)$),
we introduce $\mathcal{L}^r(M)$, the set of all ideals $\fa$ for
which $\H_\fa^i(M)$ is not artinian for some $i\geq r$. It is
evident that if $d> 0$ then $\mathcal{L}^r(M)$ is not empty. We show
that any maximal element $\fq$ of $\mathcal{L}^r(M)$ is a prime
ideal and that all Bass numbers of $\H_\fq^i(M)$ are finite for all
$i\geq r$. We conclude that this statement generalizes
\cite[Corollary 2]{DM} (see Theorem \ref {4--7} and its comment).


\section{Arithmetic of cohomological dimensions}

Assume that $\fa, \fb$ are ideals of $R$ and that $X$ is an
$R$--module. In this section, we study relationships between the
numbers $\cd(\fa, X), \cd(\fb, X), \cd(\fa+\fb, X), \cd(\fa\cap\fb, X)(= \cd(\fa\fb,
X))$, $\ara(\fa)$, etc, which are interesting in
themselves and we use them to determine artinian-ness and
non-artinian-ness of certain local cohomology modules in the next
sections.



\begin{lem} \label {2--1} Let $X$ be an $R$--module and let $t$ be a non-negative
integer such that for all $r$, $0\leq r\leq t$, $\emph\H^{t-r}_{\fa}(\emph\H^{r}_{\fb}(X))= 0$.
Then $\emph\H^{t}_{\fa + \fb}(X)$ is also zero.

\begin{proof} By \cite [Theorem 11.38]{Rot}, there is a Grothendieck spectral sequence
$$E^{p, q}_{2}:=\H^{p}_{\fa}(\H^{q}_{\fb}(X))
_{\stackrel{\Longrightarrow}{p}} \H^{p +q}_{\fa + \fb}(X).$$ For all
$r$, $0\leq r\leq t$, we have $E_{\infty}^{t-r, r}=E_{t+2}^{t-r, r}$
since $E_{i}^{t-r-i, r+i-1}= 0= E_{i}^{t-r+i, r+1-i}$ for all $i\geq
t+2$. Note that $E_{t+2}^{t-r, r}$ is a subquotient of $E_{2}^{t-r,
r}$ which is zero by assumption. Thus $E_{t+2}^{t-r, r}$ is zero,
that is $E_{\infty}^{t-r, r}= 0$. There exists a finite filtration
$$0= \phi^{t+1}H^{t}\subseteq \phi^{t}H^{t}\subseteq \cdots\subseteq \phi^{1}H^{t}\subseteq \phi^{0}H^{t}=
\H^{t}_{\fa + \fb}(X)$$ such that $E_{\infty}^{t-r,
r}=\phi^{t-r}H^{t}/\phi^{t-r+1}H^{t}$ for all $r$, $0\leq r\leq t.$
Therefore we have
$$0= \phi^{t+1}H^{t}= \phi^{t}H^{t}= \cdots= \phi^{1}H^{t}= \phi^{0}H^{t}= \H^{t}_{\fa + \fb}(X)$$
as desired.
\end{proof}
\end{lem}



The following corollary is the first application of the above lemma.

\begin{cor} \label {2--2} For a finite
$R$--module $M$, the following statements hold true.
               \begin{enumerate}
                   \item[\emph{(i)}]{$\emph{\cd}(\fa+\fb,M)\leq \emph{\cd}(\fa,M)+ \emph{\cd}(\fb,M)$.}
                   \item[\emph{(ii)}]{$\emph{\cd}(\fa\cap \fb,M)\leq \emph{\cd}(\fa,M)+ \emph{\cd}(\fb,M)$.}
                   \item[\emph{(iii)}]{$\emph{\cd}(\fa,M)\leq \displaystyle\sum_{\fp \in \emph{\Min}(\fa)}
                   \emph{\cd}(\fp,M)$.}
               \end{enumerate}

\begin{proof}(i) Assume that
$t$ is a non-negative integer such that $t> \cd(\fa,M)+ \cd(\fb,M)$.
We show that $\H^{t-r}_{\fa}(\H^{r}_{\fb}(M))= 0$ for all $r$,
$0\leq r\leq t$. If $r> \cd(\fb,M)$, then
$\H^{t-r}_{\fa}(\H^{r}_{\fb}(M))= 0$ by the definition of
cohomological dimension. Otherwise, $t-r> \cd(\fa,M)$. Since
$\Supp_{R}(\H^{r}_{\fb}(M))\subseteq \Supp_{R}(M)$,
$\cd(\fa,M)\geqslant \cd(\fa,\H^{r}_{\fb}(M))$ (see \cite [Theorem
1.4]{DY}). Therefore $\H^{t-r}_{\fa}(\H^{r}_{\fb}(M))= 0$. Now
applying Lemma \ref {2--1}, we see that $\H^{t}_{\fa + \fb}(M)=0$
which yields the assertion.

(ii) Consider the Mayer-Vietoris exact sequence
$$0\longrightarrow \Gamma_{\fa+\fb}(M)\longrightarrow \Gamma_{\fa}(M)\oplus \Gamma_{\fb}(M)
\longrightarrow \Gamma_{\fa\cap \fb}(M)\longrightarrow \cdots$$
$$\cdots\longrightarrow \H^{t}_{\fa+\fb}(M)\longrightarrow  \H^{t}_{\fa}(M)\oplus
\H^{t}_{\fb}(M)\longrightarrow \H^{t}_{\fa\cap \fb}(M)\longrightarrow
\H^{t+1}_{\fa+\fb}(M)\longrightarrow \cdots$$ and use part (i).

(iii) As $\sqrt{\fa}=\displaystyle\bigcap_{\fp\in\tiny\Min(\fa)}\fp$,
the claim follows from part (ii).
\end{proof}
\end{cor}



\begin{rem} \label {2--3} \emph{In the above corollary, one may state more precise
statements is certain cases as follows:
                   \begin{itemize}
                       \item[(ii$'$)] If $\cd(\fa, M)> 0$ and $\cd(\fb, M)> 0$, then $\cd(\fa\cap\fb, M)\leq \cd(\fa, M)+ \cd(\fb, M)-1$.
                        \item[(iii$'$)] If $R$ is local and $M$ is not $\fa$--torsion, then $$\cd(\fa, M)\leq \displaystyle\sum_{\fp\in\tiny\Min(\fa)} \cd(\fp,M)- |\Min(\fa)|+1.$$
                   \end{itemize}
Note that the proof of (ii$'$) is similar to that of  Corollary
\ref{2--2}(ii). For (iii$'$), we have $\cd(\fp, M)> 0$ for all
$\fp\in\Min(\fa)$ since $M$ is not $\fa$--torsion. The result
follows by induction on $|\Min(\fa)|$.}
\end{rem}



For a general module $X$, not necessarily finite, we have the
following result.

\begin{cor} \label {2--4} Let $X$ be an arbitrary $R$--module. Then the following statements hold true.
               \begin{enumerate}
                    \item[\emph{(i)}]{$\emph{\cd}(\fa+\fb,X)\leq \emph{\ara}(\fa)+ \emph{\cd}(\fb,X)$.}
                    \item[\emph{(ii)}]{$\emph{\cd}(\fa\cap\fb,X)\leq \emph{\ara}(\fa)+ \emph{\cd}(\fb,X)$.}
                    \item[\emph{(iii)}]{$\emph{\cd}(\fb,X)\leq \emph{\cd}(\fa,X)+ \emph{\ara}(\fb/\fa)$
                                       whenever $\fa\subseteq \fb$.}
               \end{enumerate}

\begin{proof} The proofs of (i) and (ii) are similar to those of Corollary \ref{2--2} (i) and (ii), respectively.
For (iii), let $e= \cd(\fa,X)$ and  $f= \ara(\fb/\fa)$. There exist
$x_{1},...,x_{f}\in R$ such that $\sqrt{\fb}= \sqrt{(x_{1},...,x_{f})+ \fa}$. Now, use part (i).
\end{proof}
\end{cor}



We need some sufficient conditions for validity of the isomorphism
$\H^{s}_{\fa}(\H^{t}_{\fb}(X))\cong \H^{s+t}_{\fa + \fb}(X)$, for
given non-negative integers $s$ and $t$, which is crucial for the
rest of the paper, e.g. to determine equalities in Corollary \ref
{2--2}(i) and Corollary \ref {2--4}(i).

\begin{lem} \label {2--5} Let $X$ be an arbitrary $R$--module and let $s, t$ be non-negative
integers such that
              \begin{itemize}
                      \item[(a)] $\emph\H^{s+t-i}_{\fa}(\emph\H^{i}_{\fb}(X))= 0$ for all $i\in\{ 0, \cdots, s+t\}\setminus\{t\}$,
                      \item[(b)] $\emph\H^{s+t-i+1}_{\fa}(\emph\H^{i}_{\fb}(X))= 0$ for all $i\in\{0, \cdots, t-1\}$, and
                      \item[(c)] $\emph\H^{s+t-i-1}_{\fa}(\emph\H^{i}_{\fb}(X))= 0$ for all $i\in\{t+1, \cdots, s+t\}$.
              \end{itemize}
Then we have the isomorphism
$\emph\H^{s}_{\fa}(\emph\H^{t}_{\fb}(X))\cong \emph\H^{s+t}_{\fa + \fb}(X).$

\begin{proof} Consider the Grothendieck spectral sequence
$$E^{p, q}_{2}:= \H^{p}_{\fa}(\H^{q}_{\fb}(X))
_{\stackrel{\Longrightarrow}{p}} \H^{p+q}_{\fa+\fb}(X).$$ For all
$r\geq 2$, let $Z_{r}^{s, t}= \ker(E_{r}^{s, t}\longrightarrow
E_{r}^{s+r, t+1-r})$ and $B_{r}^{s, t}= \im(E_{r}^{s-r,
t+r-1}\longrightarrow E_{r}^{s, t})$. We have exact sequences
$$0\longrightarrow B_{r}^{s, t}\longrightarrow Z_{r}^{s, t}\longrightarrow
E_{r+1}^{s, t}\longrightarrow 0$$ and
$$0\longrightarrow Z_{r}^{s, t}\longrightarrow E_{r}^{s, t}\longrightarrow
E_{r}^{s, t}/Z_{r}^{s, t}\longrightarrow 0.$$ Since, by assumptions
(b) and (c), $E_{2}^{s+r, t+1-r}= 0= E_{2}^{s-r, t+r-1}$,
$E_{r}^{s+r, t+1-r}= 0= E_{r}^{s-r, t+r-1}$. Therefore $E_{r}^{s, t}/Z_{r}^{s, t}= 0= B_{r}^{s, t}$ which shows that $ E^{s, t}_{r}=
E^{s, t}_{r+1}$ and so
$$\H^{s}_{\fa}(\H^{t}_{\fb}(X))= E^{s, t}_{2}= E^{s, t}_{3}= \cdots= E^{s, t}_{s+t+1}=
 E^{s, t}_{s+t+2}= E^{s,t}_{\infty}.$$
There is a finite filtration
$$0= \phi^{s+t+1}H^{s+t}\subseteq \phi^{s+t}H^{s+t}\subseteq \cdots\subseteq
 \phi^{1}H^{s+t}\subseteq \phi^{0}H^{s+t}= H^{s+t}_{\fa+\fb}(X)$$
such that $E^{s+t-r, r}_{\infty}=
\phi^{s+t-r}H^{s+t}/\phi^{s+t-r+1}H^{s+t}$ for all $r$, $0\leq r\leq
s+t.$

Note that for each $r$, $0\leq r\leq t- 1$ or $t+1\leq r\leq s+ t$,
$E^{s+t-r, r}_{\infty}= 0$ by assumption (a). Therefore we get
$$0= \phi^{s+t+1}H^{s+t}= \phi^{s+t}H^{s+t}= \cdots= \phi^{s+2}H^{s+t}= \phi^{s+1}H^{s+t}$$
and
$$\phi^{s}H^{s+t}= \phi^{s-1}H^{s+t}= \cdots= \phi^{1}H^{s+t}= \phi^{0}H^{s+t}= \H^{s+t}_{\fa+\fb}(X).$$
Hence $\H^{s}_{\fa}(\H^{t}_{\fb}(X))= E^{s, t}_{\infty}=
\phi^{s}H^{s+t}/\phi^{s+1}H^{s+t}= \H^{s+t}_{\fa+\fb}(X)$ as
desired.
\end{proof}
\end{lem}



Now, we are able to discuss conditions under which the inequalities
Corollary \ref{2--2}(i) and Corollary \ref{2--4}(i) become equalities.

\begin{cor} \label {2--6} Suppose that $M$ is a finite
$R$--module such that $(\fa+ \fb) M\neq M$. Then the following statements hold true.
               \begin{enumerate}
                   \item[\emph{(i)}]{$\emph\H^{\emph{\cd}(\fa,M)+ \emph{\cd}(\fb,M)}_{\fa+ \fb}(M)
                   \cong \emph\H^{\emph{\cd}(\fa,M)}_{\fa}(\emph\H^{\emph{\cd}(\fb,M)}_{\fb}(M))$.}
                   \item[\emph{(ii)}]{The following statements are
                   equivalent.
                       \begin{enumerate}
                            \item[\emph{(a)}]{$\emph{\cd}(\fa+ \fb,M)= \emph{\cd}(\fa,M)+ \emph{\cd}(\fb,M)$.}
                            \item[\emph{(b)}]{$\emph{\cd}(\fa,M)= \emph{\cd}(\fa,\emph\H^{\emph{\cd}(\fb,M)}_{\fb}(M))$.}
                            \item[\emph{(c)}]{$\emph{\cd}(\fb,M)= \emph{\cd}(\fb,\emph\H^{\emph{\cd}(\fa,M)}_{\fa}(M))$.}
                       \end{enumerate}}
               \end{enumerate}
\end{cor}

\begin{proof}
(i) Apply Lemma \ref {2--5} with $s= \cd(\fa,M)$ and $t= \cd(\fb,M)$.

(ii) The implications $(a \Rightarrow b)$ and $(a \Rightarrow c)$ are clear
from part (i) and \cite [Theorem 1.4]{DY}. For implications  $(b \Rightarrow a)$
and $(c \Rightarrow a)$, one may use part (i) and
Corollary \ref{2--2}(i).
\end{proof}



With a similar argument, one has the following result for an
arbitrary module.

\begin{cor} \label {2--7} Suppose that $X$ is an arbitrary
$R$--module. Then we have
               \begin{enumerate}
                   \item[\emph{(i)}]{$\emph\H^{\emph{\ara}(\fa)+ \emph\cd(\fb,X)}_{\fa+ \fb}(X)\cong \emph\H^{\emph{\ara}(\fa)}_{\fa}(\emph\H^{\emph{\cd}(\fb,X)}_{\fb}(X))$.}
                   \item[\emph{(ii)}]{The following statements are equivalent.
                       \begin{enumerate}
                            \item[\emph{(a)}]{$\emph{\cd}(\fa+ \fb,X)= \emph{\ara}(\fa)+ \emph\cd(\fb,X)$.}
                            \item[\emph{(b)}]{$\emph{\ara}(\fa)= \emph{\cd}(\fa,\emph\H^{\emph{\cd}(\fb,X)}_{\fb}(X))$.}
                       \end{enumerate}}
               \end{enumerate}
\end{cor}


\section{artinian local cohomology modules}




In this section, we study artinian property of local cohomology
modules. For this purpose, for ideals $\fb\supseteq\fa$, we
introduce the notion of cohomological dimension of an $R$--module
$X$ with respect to $\fb/\fa$.

\begin{defn} \label {3--1} Let
$\fb\supseteq\fa$ be ideals of $R$ and let $X$ be an $R$--module.
Define the  cohomological dimension of $X$ with respect to $\fb/\fa$
as $$\emph\cd(\fb/\fa, X):= \inf \{\emph{\cd}(\fc, X): \fc \ \emph{is
an ideal of}\  R \ \emph{and}\  \sqrt{\fb}= \sqrt{\fc+ \fa}\}.$$
\end{defn}

It is easy to see that $\cd(\fb/\fa, X)\leq \ara(\fb/\fa)$ and, for
a finite $R$--module $M$,
$$\cd(\fb/\fa, M)\geq\cd(\fb, M)-\cd(\fa, M)$$
by Corollary \ref {2--2}(i). Note that when $\fa X= 0$,
we have $\cd(\fb/\fa, X)= \cd(\fb, X)= \cd_{R/\fa}(\fb/\fa, X)$. One may
notice that if
$\Supp_R(X)\subseteq\Supp_R(M)$, then $\cd(\fb/\fa, X)\leq
\cd(\fb/\fa, M)$.

Now, we can state the following theorem.



\begin{thm} \label {3--2}  Let $\fb\supseteq\fa$ be ideals of $R$, let $X$ be an arbitrary
$R$--module and let $n$ be a non-negative integer.
\begin{enumerate}
      \item[\emph{(i)}]{If $\emph{\H}^{i}_{\fa}(X)$ is $\fb$--torsion for all $i$, $0\leq i\leq n$,
       then $\emph{\H}^{i}_{\fa}(X)\cong \emph{\H}^{i}_{\fb}(X)$ for all $i$, $0\leq i\leq n$.}
      \item[\emph{(ii)}]{If $\emph{\H}^{i}_{\fa}(X)$ is $\fb$--torsion for all $i\geq n$,
      then $\emph{\H}^{i}_{\fa}(X)\cong \H^{i}_{\fb}(X)$ for all $i\geq n+ \emph\ara(\fb/\fa)$.}
      \item[\emph{(iii)}]{Assume that $M$ is a finite $R$--module and that $\emph{\H}^{i}_{\fa}(M)$ is $\fb$--torsion for all $i\geq
      n$. Then $\emph{\H}^{i}_{\fa}(M)\cong \H^{i}_{\fb}(M)$ for all $i\geq n+
\emph\cd(\fb/\fa, M)$.}
     \end{enumerate}

\begin{proof}
Let  $u= \ara(\fb/\fa)$ and $v= \cd(\fb/\fa, M)$. There exist
$x_{1},...,x_{u}\in R$ and an ideal $\fc$ of $R$ such that $\cd(\fc,
M)= v$ and $\sqrt{(x_{1},...,x_{u})+ \fa}= \sqrt{\fb}= \sqrt{\fc+
\fa}$. In computing local cohomology modules, we may assume that
$(x_{1},...,x_{u})+ \fa= \fb= \fc+ \fa$. Now, for all $i$,  $0\leq
i\leq n$, (respectively, $i\geq n+ u$, $i\geq n+ v$,) apply Lemma \ref
{2--5} with $s= 0$ and $t= i$ to obtain the isomorphisms
$\Gamma_{(x_{1},...,x_{u})}(\H^i_\fa(X))\cong \H^i_\fb(X)$ for all
$i$, $0\leq i\leq n$, (respectively,
$\Gamma_{(x_{1},...,x_{u})}(\H^i_\fa(X))\cong \H^i_\fb(X)$ for all
$i\geq n+ u$, $\Gamma_{\fc}(\H^i_\fa(M))\cong \H^i_\fb(M)$ for all
$i\geq n+ v$,). Therefore all of the assertions follow.
\end{proof}
\end{thm}



\begin{cor} \label {3--3} Let $R$ be a local ring with maximal
ideal $\fm$, let $M$ be a finite $R$--module and let $n$ be a
non-negative integer. If $\emph\dim_{R} (\H^{i}_{\fa}(M))\leq 0$ for
all $i$, $0\leq i\leq n$ \emph{(}respectively, for all $i\geq
n$\emph{)}, then $\H^{i}_{\fa}(M)$ is artinian for all $i$, $0\leq
i\leq n$ \emph{(}respectively, for all $i\geq
n+ \emph\cd(\fm/\fa, M)$\emph{)}.

\begin{proof} Since $\H^{i}_{\fm}(M)$ is artinian for all $i$, the assertion
follows from Theorem \ref {3--2}.
\end{proof}
\end{cor}



Recall that a finite $R$--module $M$ over a local ring $(R, \fm)$ is
called a {\it generalized Cohen--Macaulay} module if $\H_\fm^i(M)$
is of finite length for all $i< \dim_R(M)$. The following result
gives us a characterization for a finite module $M$ over a local
ring to be generalized Cohen-Macaulay in terms of the existence of
an ideal $\fa$ for which $\H^{i}_{\fa}(M)$ is of finite length for
all $i< \dim_{R}(M)$.

\begin{cor} \label {3--4} Let $R$ be a local ring with maximal
ideal $\fm$ and let $M$ be a finite $R$--module. Then the following
statements are equivalent.
                       \begin{enumerate}
                            \item[\emph{(i)}]{$M$ is generalized Cohen-Macaulay module.}
                            \item[\emph{(ii)}]{There exists an ideal $\fa$ such that $\emph\H^{i}_{\fa}(M)$
                                        is of finite length for all $i$, $0 \leq i < \emph{\dim}_{R}(M)$.}
                       \end{enumerate}

\begin{proof}  (i) $\Rightarrow$ (ii). It is trivial.

(ii) $\Rightarrow$ (i). This follows from Theorem \ref {3--2}(i).
\end{proof}
\end{cor}



A non-zero $R$--module $X$ is called {\it secondary} if its
multiplication map by any element $a$ of $R$ is either surjective or
nilpotent. A prime ideal $\fp$ of $R$ is said to be an {\it attached
prime} of $X$ if $\fp =(T :_{R} X)$ for some submodule $T$ of $X$.
If $X$ admits a reduced secondary representation, $X=X_1+X_2+\dots
+X_n$, then the set of attached primes $\Att_R(X)$ of $X$ is equal
to $\{ \sqrt{0:_R X_i}: i= 1,\cdots, n\}$ (cf. \cite{Mac}).

Assume that $M$ is a finite $R$--module of finite dimension $d$ and
that $\fa$ is an ideal of $R$. It is well-known that $\H_\fa^d(M)$
is artinian. If $(R, \fm)$ is local, then the first author and
Yassemi in \cite [Theorem A] {DY1} (see also \cite [Theorem 8.2.1]
{He}) showed that $\Att_R(\H_\fa^d(M))= \{\fp\in\Assh_R(M):
\H_\fa^d(R/\fp)\not= 0\}$ which generalized the well-known result
$\Att_R(\H_\fm^d(M))= \Assh_R(M) (= \{\fp\in\Supp_R(M): \dim(R/\fp)=
d\})$ (see \cite[Theorem 2.2]{MS}). In the following remark, those
ideals $\fa$ for which $\Att_R(\H_\fa^d(M))= \Assh_R(M)$ are
characterized. Denote the height support, $\hSupp_R(M)$, of $M$ as
the set of all $\fp\in\Supp_R(M)$ such that $\fp\in\V(\fq)$ for some
$\fq\in\Assh_R(M)$.

\begin{rem} \label {3--5} \emph{Let $(R, \fm)$ be a complete local ring and let $M$ be a non-zero finite $R$--module with Krull
dimension $d$. Then the following statements are equivalent.
             \begin{itemize}
                     \item[(i)]   $\H^{d}_{\fa}(M)\cong \H^{d}_{\fm}(M)$.
                     \item[(ii)]  $\Att_R(\H_\fa^d(M))= \Assh_R(M)$.
                     \item[(iii)] $\V(\fa)\cap \hSupp_R(M)= \{\fm\}$.
             \end{itemize}}
\end{rem}

The proof of (i) $\Rightarrow$ (ii) is clear. To prove (ii)
$\Rightarrow$ (iii), one may use  Lichtenbaum-Hartshorne Vanishing
Theorem. For (iii) $\Rightarrow$ (i), choose a submodule $N$ of $M$
such that $\Ass_R(N)= \Ass_R(M)\setminus \Assh_R(M)$ and
$\Ass_R(M/N)= \Assh_R(M)$ to obtain $\H_\fa^d(M)\cong \H_\fa^d(M/N)$
and $\H_\fm^d(M)\cong \H_\fm^d(M/N)$. Therefore
$\Supp_R(\H_\fa^i(M/N))\subseteq \{\fm\}$ for all $i$. Applying
Theorem \ref {3--2}(i) gives the claim. This remark shows that if
$M$ is equidimensional, then $\Att_R(\H_\fa^d(M))\not= \Assh_R(M)$
for each ideal $\fa$ with $\h_M(\fa)< \dim_R(M)$.\\



Recall that, an $R$--module $X$ is said to be {\it minimax} if it
has a finite submodule $X'$ such that $X/X'$ is artinian (See \cite
{Z}). Note that the class of minimax modules includes all finite and
all artinian modules. We close this section by showing that if $\fm$
is a maximal ideal containing $\fa$, then $\H_\fm^i(M)$ is artinian
for all $i\leq n$ (respectively, for all $i\geq n+\cd(\fm/\fa, M)$)
whenever $\H_\fa^i(M)$ is minimax for all $i\leq n$ (respectively,
for all $i\geq n$). We first bring an analogous lemma as Lemma \ref
{2--1}.

\begin{lem} \label {3--6} Let $X$ be an $R$--module and let $t$ be a non-negative
integer such that $\emph\H^{t-r}_{\fa}(\emph\H^{r}_{\fb}(X))$ is artinian for all $r$, $0\leq r\leq t$.
Then $\emph\H^{t}_{\fa + \fb}(X)$ is artinian.

\begin{proof} By the Grothendieck spectral sequence
$$E^{p, q}_{2}:=\H^{p}_{\fa}(\H^{q}_{\fb}(X)) _{\stackrel{\Longrightarrow}{p}} \H^{p +q}_{\fa + \fb}(X),$$
the proof is similar to that of Lemma \ref {2--1}.
\end{proof}
\end{lem}




\begin{thm} \label {3--7}  Let $\fm$ be a maximal ideal of $R$ contains $\fa$, let $X$ be an arbitrary
$R$--module and let $n$ be a non-negative integer. Then
    \begin{itemize}
        \item[(i)]{If $\emph{\H}^{i}_{\fa}(X)$ is minimax for all $i$, $0\leq i\leq n$,
                    then $\emph{\H}^{i}_{\fm}(X)$ is artinian for all $i$, $0\leq i\leq n$.}
        \item[(ii)]{If $\emph{\H}^{i}_{\fa}(X)$ is minimax for all $i\geq n$,
                    then $\emph{\H}^{i}_{\fm}(X)$ is artinian for all $i\geq n+ \emph\ara(\fm/\fa)$.}
        \item[(iii)]{Assume that $M$ is a finite $R$--module and that $\emph{\H}^{i}_{\fa}(M)$ is minimax for all $i\geq n$.
                    Then $\emph{\H}^{i}_{\fm}(M)$ is artinian for all $i\geq n+\emph\cd(\fm/\fa, M)$.}
     \end{itemize}

\begin{proof}
By considering lemma \ref {3--6}, this is similar to that of Theorem
\ref {3--2}.
\end{proof}
\end{thm}


\section{Non-artinian local cohomology modules}




In this section, we study those local cohomology modules which are
not artinian. The following two results give us many non-artinian
local cohomology modules.

\begin{cor} \label {4--1} Let $X$ be an
$R$--module, let $n$ be a positive integer and let
$x_{1},...,x_{n}\in R$ such that $\emph{\cd}(\fa+
(x_{1},...,x_{n}),X)= \emph{\cd}(\fa,X)+ n$. Then
$\emph{\dim}_R(\emph\H^{\emph{\cd}(\fa,X)}_{\fa}(X))\geq n$. In
particular, $\emph\H^{\emph{\cd}(\fa,X)}_{\fa}(X)$ is not artinian.

\begin{proof} By Corollary \ref {2--4}(i),
$\ara(x_{1},...,x_{n})= n$. By Corollary \ref {2--7}(ii) and
Grothendieck Vanishing Theorem, we have
$\dim_{R}(\H^{{\tiny\cd(\fa,X)}}_{\fa}(X))\geq n$ and so
$\H^{{\tiny\cd(\fa,X)}}_{\fa}(X)$ is not artinian.
\end{proof}
\end{cor}



\begin{cor} \label {4--2} \emph{(cf.} \cite [Proposition 3.2]{BN}\emph{)}
Let $(R, \fm)$ be a local ring and let $M$ be a finite $R$--module
with Krull dimension $d$. Assume also that $\fa$ is generated by a
subset of system of parameters $x_{1}, ..., x_{n}$ of $M$ of length
$n$. Then $\emph{\dim}_{R}(\emph\H^{\emph{\cd}(\fa,M)}_{\fa}(M))=
d-n$. In particular, if $n< d$, then
$\emph\H^{\emph{\cd}(\fa,M)}_{\fa}(M)$ is not artinian.

\begin{proof} There exist $x_{n+1}, ..., x_{d}\in R$ such that
$x_{1}, ..., x_{d}$ is a  system of parameters of $M$. Set $\fb=
(x_{n+1}, ...,x_{d})$. As $\fm= \sqrt{\fa+ \fb+ \Ann_{R}(M)}$, we
can, and do, assume that $\fa+ \fb= \fm$. By Corollary \ref
{2--2}(i), $\cd(\fa,M)= n$ and $\cd(\fb,M)= d-n$. Now, by using
Corollary \ref {2--6}(ii), we obtain
$\dim_{R}(\H^{n}_{\fa}(M))\geqslant d-n$. On the other hand, we have
$\dim_{R}(\H^{n}_{\fa}(M))\leqslant d-n$ since
$\Supp_{R}(\H^{n}_{\fa}(M))\subseteq \Supp_{R}(M/\fa M).$ Thus
$\dim_{R}(\H^{n}_{\fa}(M))= d-n$ as desired.
\end{proof}
\end{cor}



Now it is natural to raise the following question.

\begin{ques} \label {4--3} Assume that $M$ is a finite $R$--module and that
$\emph{\H}_\fa^{\emph\cd(\fa, M)}(M)$ is not artinian. Is there an element
$x$ in $R$ such that $$\emph\cd(\fa+Rx, M)=\emph\cd(\fa, M)+ 1?$$
\end{ques}

It is clear that the above question has a positive answer if $R$ is
local and $\fa$ is generated by a subset of system of parameters of
$M$ of length smaller than $\dim_R(M)$.\\



In the rest of the paper, we study the set of ideals $\fb$ of $R$
such that $\H_\fb^i(M)$ is not artinian for some non-negative
integer $i$.

\begin{defn} \label {4--4}  Let $M$ be a finite $R$--module and let $r$ be a non-negative
integer. Define the set of ideals
\begin{center}
$\mathcal{L}^r(M):= \{\fb: \H_\fb^i(M)$ is not artinian for some
$i\geq r$\}.
\end{center}
\end{defn}

Note that $\mathcal{L}^r(M)$ is the empty set for all $r\geq
\dim_R(M)$. If $0\leq r< \dim_R(M)$,  $\mathcal{L}^r(M)$ is
non-empty by Corollary \ref{4--2}. In the following remark, it is
shown that the set $\mathcal{L}^r(M)$ is independent of the module
structure.



\begin{rem} \label {4--5} Assume that $L$, $M$ and $N$ are
finite $R$--modules and that $r$ is a non--negative integer.
Then the following statements are true.
         \begin{itemize}
         \item[(i)] If $\emph\Supp_R(N)\subseteq \emph\Supp_R(M)$, then $\mathcal{L}^r(N)\subseteq \mathcal{L}^r(M)$.
                \item[(ii)] If $\ 0\longrightarrow L\longrightarrow M\longrightarrow N\longrightarrow 0$
                 is an exact sequence,
                            then $\mathcal{L}^r(M)= \mathcal{L}^r(L)\cup \mathcal{L}^r(N)$.
                \item[(iii)] $\mathcal{L}^r(M)= \displaystyle\bigcup_{\fp\in \emph\Ass_R(M)}\mathcal{L}^r(R/\fp)$.
         \end{itemize}
\end{rem}

\begin{proof} (i) Assume that $\fa$ is an ideal of $R$ which is not
in $\mathcal{L}^r(M)$; so that $\H^{i}_{\fa}(M)$ is artinian for all
$i\geq r$. Therefore $\H^{i}_{\fa}(N)$ is artinian for all $i\geq r$
by \cite [Theorem 3.1]{AM}, that is $\fa$ does not belong to
$\mathcal{L}^r(N)$. Thus $\mathcal{L}^r(N)\subseteq
\mathcal{L}^r(M)$ as desired.

(ii) By (i), $\mathcal{L}^r(M)\supseteq \mathcal{L}^r(L)\cup
\mathcal{L}^r(N)$. Assume that $\fa\in \mathcal{L}^r(M)$. There
exists an integer $i$, $i\geq r$, such that $\H_\fa^i(M)$ is not
artinian. Now, by the exact sequence $\H_\fa^i(L)\longrightarrow
\H_\fa^i(M)\longrightarrow \H_\fa^i(N)$, the other inclusion
follows.

(iii) By (i), we have the inclusion $\mathcal{L}^r(M)\supseteq
\displaystyle\cup_{\fp\in {\small\Ass}_R(M)}\mathcal{L}^r(R/\fp)$. Assume,
conversely, that $\fb\not\in \displaystyle\cup_{\fp\in
{\small\Ass}_R(M)}\mathcal{L}^r(R/\fp)$. There is a prime filtration
$0= M_0\subset  M_1\subset \cdots \subset M_s= M$ of $M$ such that,
for all $j\in \{1, \cdots, s\}$, $M_j/M_{j-1}\cong R/\fp_j$ for some
$\fp_j\in \Supp_R(M)$. For each $j\in \{1, \cdots, s\}$, there is
$\fq_{j}\in \Ass_R(M)$ contained in $\fp_j$ and thus, by assumption
and part (i), $\fb\not\in \mathcal{L}^r(R/\fp_j)$. Now, by applying
$\H_\fb^i(-)$ on each exact sequence
$$0\longrightarrow M_j\longrightarrow M_{j+1}\longrightarrow
M_{j+1}/M_j\longrightarrow 0,$$ it follows that $\fb\not\in
\mathcal{L}^r(M)$.
\end{proof}



Before bringing the main theorem of this section, recall the
following result which is straightforward from the fact that, for an
$R$--module $X$ and for each $\alpha\in R$, the kernel (respectively,
the cokernel) of the natural map $X\longrightarrow X_\alpha$ is
$\H_{R\alpha}^0(X)$ (respectively, $\H_{R\alpha}^1(X)$), where
$X_\alpha$ denote the localization of $X$ at set $\{1, \alpha,
\alpha^2, \alpha^3,\cdots\}$.

\begin{prop} \label {4--6} For any $R$--module $X$ and for any $\alpha\in R$, there are exact
sequences
$$0\longrightarrow \emph\H_{R\alpha}^1(\emph\H_\fa^{i-1}(X))\longrightarrow
\emph\H_{\fa+R\alpha}^{i}(X)\longrightarrow
\emph\H_{R\alpha}^0(\emph\H_\fa^i(X))\longrightarrow 0,$$
for all $i\geq 0$.
\end{prop}

\begin{proof}
See \cite[Proposition 8.1.2]{BS} (see also \cite[Theorem 2.5]{BFT}).
\end{proof}



The $i$th Bass number of $X$ with respect to the prime ideal $\fp$ of $R$, denoted by
$\mu^i(\fp, X)$, is defined to be the number of copies of the
indecomposable injective module $\E_R(R/\fp)$ in the direct sum
decomposition of the $i$th term of a minimal injective resolution of
$X$, which is equal to the rank of the vector space
$\Ext_{R_\fp}^i(k(\fp), X_\fp)$ over the field $k(\fp)= R_\fp/\fp
R_\fp$. When $(R, \fm)$ is local, we write $\mu^i(X):= \mu^i(\fm,
X)$ and refer it the $i$th Bass number of $X$.

In the following theorem, we study Bass numbers of certain
non--artinian local cohomology modules.

\begin{thm} \label {4--7} Assume that $(R, \fm)$ is a local ring and that $M$ is a finite $R$--module
with Krull dimension $d$. Let $r< d$ be a fixed non-negative
integer. Then for each maximal element $\fq$ of the non-empty set
$\mathcal{L}^r(M)$,
          \begin{itemize}
               \item[(i)] $\mu^j(\emph\H_\fq^i(M))< \infty$ for all $j\geq 0$ and all $i\geq r$.
               \item[(ii)] $\fq$ is a prime ideal.
          \end{itemize}

\begin{proof}
(i) As $\H_\fm^i(M)$ is artinian for all $i\geq 0$, we have
$\fq\not=\fm$. Choose an element $x\in\fm\setminus \fq$. Thus
$\H_{\fq+Rx}^i(M)$ is artinian for all $i\geq r$. Using the exact
sequence
$$0\longrightarrow \H_{Rx}^1(\H_\fq^{i-1}(M))\longrightarrow
\H_{\fq+Rx}^i(M)\longrightarrow
\H_{Rx}^0(\H_{\fq}^i(M))\longrightarrow 0,$$ it follows that, for
each $i\geq r$, the modules $\H_{Rx}^1(\H_\fq^{i}(M))$ and
$\H_{Rx}^0(\H_{\fq}^i(M))$ are artinian and so they have finite Bass
numbers. It follows by \cite[Theorem 2.1]{DY2} that
$\mu^j(\H_\fq^i(M))< \infty$ for all $j\geq 0$ and all $i\geq r$.

(ii) Assume that $x, y\in\fm\setminus \fq$ such that $xy\in\fq$. As
$\fq+ Rx$ and $\fq+ Ry$  properly contain $\fq$, it follows that the
modules $\H_{\fq+Rx}^i(M)$, $\H_{\fq+Ry}^i(M)$, and $
\H_{\fq+Rx+Ry}^i(M)$ are artinian for all $i\geq r$. Applying the
Mayer-Vietoris exact sequence
$$\H_{\fq+Rx}^i(M)\oplus \H_{\fq+Ry}^i(M)
\longrightarrow \H_{(\fq+Rx)\cap(\fq+Ry)}^i(M)\longrightarrow
\H_{\fq+Rx+Ry}^{i+1}(M),$$ we find that
$\H_{(\fq+Rx)\cap(\fq+Ry)}^i(M)$ is artinian for $i\geq r$. Note
that
$$\begin{array}{llll}
\sqrt{\fq}&\subseteq
\sqrt{(\fq+Rx)\cap(\fq+Ry)}\\&=\sqrt{(\fq+Rx)(\fq+Ry)}\\&=\sqrt{\fq^2+\fq
x+\fq y+Rxy}\\&\subseteq \sqrt{\fq}.
\end{array}$$
and hence $\H_{(\fq+Rx)\cap(\fq+Ry)}^i(M)\cong \H_\fq^i(M)$ is
artinian for $i\geq r$. This contradicts the fact that $\fq\in
\mathcal{L}^r(M)$, and so $\fq$ is a prime ideal.
\end{proof}
\end{thm}

There have been many attempts in the literature made to find some
conditions for the ideal $\fa$ to have finiteness for the Bass
numbers of the local cohomology modules supported at $\fa$. In
\cite[Corollary 2]{DM}, Delfino and Marley showed that the Bass
number $\mu^i(\fp, \H_\fa^j(M))$ is finite for all $\fp\in\Spec R$
and all $i, j$ whenever $M$ is a finite module over a ring $R$ and
$\fa$ is an ideal of $R$ with $\dim R/\fa= 1$.

Assume that $\fa$ and $\fb$ are two ideals of a local ring $(R,
\fm)$ with $\dim(R/\fa)= \dim(R/\fb)= 1$ such that $\V(\fa+ \fb)=
\{\fm\}$. Write the Mayer-Vietoris exact sequence
$$ \H^{j}_{\fm}(M)\longrightarrow  \H^{j}_{\fa}(M)\oplus
\H^{j}_{\fb}(M)\longrightarrow \H^{j}_{\fa\cap
\fb}(M)\longrightarrow \H^{j+1}_{\fm}(M).$$ As $\H^{i}_{\fm}(M)$ is
artinian for all $i$, we find that $\H^{j}_{\fa\cap \fb}(M)$ has
finite Bass numbers if and only if both $\H^{j}_{\fa}(M)$ and
$\H^{j}_{\fb}(M)$ have finite Bass numbers. Therefore
\cite[Corollary 2]{DM} is equivalent to the case where the ideal
$\fa$ is prime.\\

{\bf Comment.} Assume that $\fp$ is a prime ideal of $R$ such that
$\dim(R/\fp)= 1$ and $r$ is the smallest integer (if there is any)
such that $\H_\fp^i(M)$ is not artinian. Thus $\fp$ is a maximal
element of $\mathcal{L}^{r}(M)$. By Theorem \ref{4--7},
$\mu^j(\emph\H_\fp^i(M))< \infty$ for all $j\geq 0$ and all $i\geq
r$. As $\H_\fp^i(M)$ is artinian for all $i< r$, all $\H_\fp^i(M)$
have finite Bass numbers. Thus Theorem \ref{4--7} generalizes
\cite[Corollary 2]{DM}.\\

{\bf Acknowledgement.} The authors would like to thank the referee
for the invaluable comments on the manuscript.

\bibliographystyle{amsplain}

\end{document}